\documentclass[12 pt]{amsart}

\usepackage{amsmath}
\usepackage{amsfonts}
\usepackage{amssymb}
\usepackage{graphicx}
\usepackage[margin=3cm]{geometry}
\usepackage{hyperref}
\usepackage{enumerate}
\usepackage{mathrsfs}

\usepackage{amsthm}
\usepackage{mathtools}

\newcommand{\R}{\mathbb{R}}

\newcommand{\Z}{\mathbb{Z}}

\DeclareMathOperator{\Diff}{Diff}
\DeclareMathOperator{\Homeo}{Homeo}
\DeclareMathOperator{\Fix}{Fix}

\DeclareMathOperator{\supp}{supp}
\DeclareMathOperator{\Int}{Int}

\newtheorem{theorem}{Theorem}

\newtheorem{corollary}[theorem]{Corollary}

\theoremstyle{definition}
\newtheorem{remark}[theorem]{Remark}

\theoremstyle{definition}
\newtheorem{question}[theorem]{Question}

\theoremstyle{definition}
\newtheorem{definition}[theorem]{Definition}

\usepackage[usenames,dvipsnames]{color}

\usepackage{float}

\begin{document}

\title{Smooth gluing of group actions and applications}
\author{Kiran Parkhe}


\begin{abstract}
Let $M_1$ and $M_2$ be two $n$-dimensional smooth manifolds with boundary. Suppose we glue $M_1$ and $M_2$ along some boundary components (which are, therefore, diffeomorphic). Call the result $N.$ If we have a group $G$ acting continuously on $M_1,$ and also acting continuously on $M_2,$ such that the actions are compatible on glued boundary components, then we get a continuous action of $G$ on $N$ that stitches the two actions together. However, even if the actions on $M_1$ and $M_2$ are smooth, the action on $N$ probably will not be smooth.

We give a systematic way of smoothing out the glued $G$-action. This allows us to construct interesting new examples of smooth group actions on surfaces, and to extend a result of Franks and Handel \cite{F&H} on distortion elements in diffeomorphism groups of closed surfaces to the case of surfaces with boundary.
\end{abstract}
\maketitle

\section{Introduction}
For any manifold with boundary $M,$ we denote the boundary by $\partial M.$ We assume manifolds are smooth. We write $\Homeo(M)$ and $\Diff^r(M)$ for homeomorphisms or $C^r$ diffeomorphisms of $M,$ respectively. If $G$ is a discrete group, a \textit{continuous (smooth) action} of $G$ on $M$ is a homomorphism $G \to \Homeo(M) (\Diff^\infty(M)).$

Let $M_1$ and $M_2$ be manifolds with boundary, and let $f_1$ and $f_2$ be homeomorphisms of $M_1$ and $M_2,$ respectively. Suppose we glue in pairs some diffeomorphic boundary components of $M_1$; call the result $N.$ If $f_1$ is compatible on the boundary components that we glue, we get an induced homeomorphism $\mathfrak{g}(f_1) \in \Homeo(N).$ Similarly, if we glue boundary components of $M_1$ to boundary components of $M_2,$ calling the result $N,$ and if $f_1$ and $f_2$ are compatible on corresponding boundary components, we get a homeomorphism $\mathfrak{g}(f_1, f_2) \in \Homeo(N).$ However, even if $f_1, f_2 \in \Diff^\infty(M_i),$ it does not follow that $\mathfrak{g}(f_1)$ or $\mathfrak{g}(f_1, f_2) \in \Diff^\infty(N)$: they will probably fail to be smooth across glued boundary components.

The goal of this article is to give a systematic procedure to apply to $f_1$ and $f_2,$ so that when the resulting maps are glued together we get a diffeomorphism of $N.$ The idea is simple: to make points near shared boundaries move approximately like points on the boundaries, we apply a topological conjugacy which ``crushes'' nearby points very strongly towards the shared boundaries. This has the effect of removing (up to all orders of derivatives) any motion under $f_1$ or $f_2$ that is transverse to the shared boundaries.

Our conjugacy is an example of what Kloeckner \cite{Kloeckner} calls \emph{stretching}. He considers the question of when smooth or analytic group actions which are topologically conjugated by a stretch are in fact smoothly conjugate.

Some authors have considered the gluing of group actions. Katok-Lewis \cite{K&L} consider the linear action of $SL(n, \Z)$ on $T^n.$ They blow up at the global fixed point 0, introducing a copy of $S^{n - 1}$ at that point. Since the original action was linear, two of these actions can be glued along $S^{n - 1},$ and the result is real-analytic. They even make it volume-preserving. Farb-Shalen \cite{F&S} do a similar construction in dimension 3; they consider a finite $SL(3, \Z)$-invariant set. To our knowledge, our construction smoothing out general glued group actions (not assumed to originate from linear actions) is new. The proof is based on an article of Tsuboi \cite{Tsuboi}.

Our result has important applications to the study of smooth group actions. If a group $G$ acts smoothly on $M_1$ and also acts smoothly on $M_2,$ and the actions agree on common boundary components, by applying our result we get a smooth action of $G$ on $N$; since our procedure involves conjugation, it respects the group structure. We give examples in which the group $G$ is the discrete Heisenberg group in the first ``applications'' section.

Another use of our result is to extend theorems about group actions on manifolds without boundary to the case of manifolds with boundary. The idea is, if we have a manifold with boundary $M,$ we can form the \textit{double} $D(M),$ which is gotten by taking two copies of $M$ and gluing corresponding boundary components. By our theorem, any smooth action of $G$ on $M$ yields a smooth action of $G$ on $D(M),$ and if we know something about actions of $G$ on $D(M)$ then we may be able to draw a similar conclusion for actions of $G$ on $M.$ In the second ``applications'' section, we use this idea to show that a result of Franks and Handel (\cite{F&H}) on invariant measures for distortion elements in surface diffeomorphism groups also holds for compact surfaces with boundary.

\vspace{12pt}
\noindent
\textit{Acknowledgements}. The author would like to thank his advisor, John Franks, for many useful discussions and ideas. The author also thanks Benson Farb for a helpful conversation. Finally, the author thanks Kamlesh Parwani for pointing out the necessity of checking that the theorem works when $\Diff^\infty$ is replaced with $\Diff^r.$

\section{Main Result}
Let $M_1, M_2$ be smooth $n$-manifolds with boundary. We will glue boundary components of $M_1$ and $M_2.$ Our proof also implies the corresponding result if we are identifying boundary components of a single manifold.

Let $\{(C_i)_1\},$ $\{(C_i)_2\}$ be components of $\partial M_1$ and $\partial M_2,$ respectively. There are at most countably many; they are smooth $(n - 1)$-manifolds without boundary. Suppose there exist diffeomorphisms $\alpha_i \colon (C_i)_1 \to (C_i)_2.$ Write $C_1 = \cup_i (C_i)_1$ and $C_2 = \cup_i (C_i)_2.$ Define $\alpha \colon C_1 \to C_2$ by $\alpha(x) = \alpha_i(x)$ for $x \in (C_i)_1.$ We may form a new manifold $N$ by gluing the boundaries according to these diffeomorphisms: $N = M_1\sqcup M_2/\sim,$ where $x \sim \alpha(x)$ for $x \in C_1.$ Let $\pi \colon M_1\sqcup M_2 \to N$ be the projection.

We may endow $N$ with smooth structure as follows. Away from the glued boundaries, we use the smooth structure of $M_1$ or $M_2.$ On the boundary, we rely on product neighborhoods. Specifically, choose a neighborhood $U_1$ of $C_1$ diffeomorphic to $C_1 \times [0, 1)$; let $\eta_1 \colon U_1 \to C_1 \times [0, 1)$ be a diffeomorphism with the property that for $x \in C_1, \eta_1(x) = (x, 0).$ Note that $(\alpha \times id) \circ \eta_1$ sends $U_1$ to $C_2 \times [0, 1).$ Similarly, take $U_2$ and $\eta_2$ such that $\eta_2 \colon U_2 \to C_2 \times (-1, 0]$ is a diffeomorphism and for $x \in C_2, \eta_2(x) = (x, 0).$ By putting together $(\alpha \times id) \circ \eta_1$ and $\eta_2$, we get a homeomorphism $\eta \colon \pi(U_1 \cup U_2) \to C_2 \times (-1, 1).$ We declare it to be a diffeomorphism.

Let $\mathcal{A} = \{(f_1, f_2) \colon f_1 \in \Homeo(M_1), f_2 \in \Homeo(M_2),$ and $\alpha \circ f_1|_{C_1} = f_2|_{C_2} \circ \alpha\}.$ Note that we do not require $(C_i)_j$ to be $f_j$-invariant. If $(f_1, f_2) \in \mathcal{A},$ they agree on their glued boundaries, so we get a glued map $\mathfrak{g}(f_1, f_2) \in \Homeo(N).$

\begin{theorem}
There exist homeomorphisms $\Psi_1$ and $\Psi_2$ of $M_1$ and $M_2$ with the following property. For any $f_i \in \Diff^\infty(M_i), \Psi_i^{-1}f_i\Psi_i \in \Diff^\infty(M_i).$ Furthermore, if $(f_1, f_2) \in \mathcal{A},$ then $\mathfrak{g}(\Psi_1^{-1}f_1\Psi_1, \Psi_2^{-1}f_2\Psi_2) \in \Diff^\infty(N).$
\end{theorem}

\begin{proof}
We will define $X_1 \colon C_1 \times [0, 1) \to C_1 \times [0, 1)$ below. Then $X_2 \colon C_2 \times (-1, 0]$ will be given by $X_2 = (\alpha \times -id)\circ X_1 \circ (\alpha \times -id)^{-1},$ and we will set

\begin{displaymath}
   \Psi_i(x) = \left\{
     \begin{array}{cl}
     \eta_i^{-1}(X_i(\eta_i(x))), & x \in U_i\\
       x, & x \in M_i \setminus U_i
       \end{array}
   \right.
   \end{displaymath}
   for i = 1, 2.

We will construct $X_1$ so that it satisfies the following properties:
\begin{enumerate}
\item It is of the form $X_1(x, y) = (x, \chi(y))$
\item $\chi$ is a $C^\infty$ diffeomorphism of $(0, 1)$ with $\chi(y) = \phi^2(y)$ (i.e. $\phi(\phi(y))$) in a neighborhood of $0$, where $\phi(y) = e^{-1/y}$, and $\chi(y) = y$ in a neighborhood of $1$.
\end{enumerate}

In all that follows, we will not worry about domains and codomains. This is because we are only concerned with local behavior; we could if desired specify (co)domains, but which choices we made would not affect our calculations. Technically, we are dealing with \textit{germs} of maps, but we leave this implicit to avoid cumbersome notation.

Let $p \in C_1$. Let $\Phi$ be defined near $(p, 0)$ in $C_1 \times [0, 1)$ to be $\Phi(x, y) = (x, \phi(y))$. We claim that (1) if $g$ is $C^\infty$ at $(p, 0)$, then so is $\Phi^{-1}g\Phi$. Since locally $X_1^{-1}gX_1 = \Phi^{-2}g\Phi^2$, applying (1) twice implies that $X_1^{-1}gX_1$ is also $C^\infty$ at $(p, 0)$. Let $g_x$ and $g_y$ denote the first and second components of $g$ (which is defined on some open set in $C_i \times [0, 1)$). Then we further claim that (2) $X_1^{-1}gX_1$ looks to all orders like $(x, y) \mapsto (g_x(x, 0), y)$ at $(p, 0)$. These claims follow Tsuboi \cite{Tsuboi}.

Let us first show why claims (1) and (2) are enough to establish the theorem. Let $(f_1, f_2) \in \mathcal{A}.$ The map $\eta_1f_1\eta_1^{-1}$ is $C^\infty$ at $(p, 0),$ and $X_1^{-1}\eta_1f_1\eta_1^{-1}X_1$ looks like $(x, y) \mapsto ((\eta_1f_1\eta_1^{-1})_x(x, 0), y)$ to all orders at $(p, 0).$ Similarly, $\eta_2f_2\eta_2^{-1}$ is $C^\infty$ at $(\alpha(p), 0),$ and $X_2^{-1}\eta_2f_2\eta_2^{-1}X_2$ looks like $(x, y) \mapsto ((\eta_2f_2\eta_2^{-1})_x(x, 0), y)$ to all orders at $(\alpha(p), 0).$ Therefore, when we glue $(\alpha \times id)X_1^{-1}\eta_1f_1\eta_1^{-1}X_1(\alpha \times id)^{-1}$ and $X_2^{-1}\eta_2f_2\eta_2^{-1}X_2,$ the result is $C^\infty.$ The fact that $\mathfrak{g}(f_1, f_2) \in \Diff^\infty(N)$ follows, since we declared $\eta$ to be a diffeomorphism.

Let us consider claim (1). We must show that $\Phi^{-1}g\Phi$ is $C^\infty$. The only difficulty is that $\Phi^{-1}$ is not differentiable when $y = 0$ (it stretches too strongly there). The first component of $\Phi^{-1}g\Phi$ -- the one in $C_i$, which we may denote $(\Phi^{-1}g\Phi)_x$ -- is not affected by $\Phi^{-1}$, so it is automatically $C^\infty$.

Consider the second component, $(\Phi^{-1}g\Phi)_y$. Since $g_y(x, 0)$ vanishes for all $x \in C_i$, by Taylor's theorem $\frac{g_y}{y}$ extends continuously to $y = 0$; it is $C^\infty$ and nonvanishing at $y = 0$. Let us denote this by $h$, so $g_y = y\cdot h$. Then

\begin{displaymath}
   \begin{array}{lll}
       (\Phi^{-1}g\Phi)_y & = & -1/\log(g_y(x, \phi(y)))\\
       \noalign{\medskip}
       & = & -1/\log(\phi(y)h(x, \phi(y)))\\
       \noalign{\medskip}
       & = & y/(1 - y\log(h(x, \phi(y)))).
     \end{array}
\end{displaymath}

Therefore, $(\Phi^{-1}g\Phi)_y$ is $C^\infty$. So we have established claim (1).

Now we consider claim (2). If $\bar{g}$ denotes the restriction of $g$ to $C_1 \times \{0\}$, then we must show that $\Phi^{-2}g\Phi^2$ and $\bar{g} \times Id$ are $C^\infty$-close at $(p, 0)$. Introducing local coordinates about $p$ and about $\bar{g}(p)$ in $C_1$ which send $p$ and $\bar{g}(p)$ respectively to the origin in $\R^{n - 1}$, we can assume we are looking at maps (locally defined around the origin) of $\R^{n - 1} \times [0, 1)$ (points of which can be denoted $(\vec{x}, y)$, where $\vec{x} = (x_1, \dots, x_{n - 1})$). In this context, we want to show that $\Phi^{-2}g\Phi^2 - \bar{g} \times Id \eqqcolon G$ has all partial derivatives of all orders equal to $0$ at the origin. (Note that it does not matter which local coordinates we chose.)

We use the following notation: subscript $y$ means the $y$-component as before; subscript $i$ for $1 \leq i \leq n - 1$ means the $x_i$-component. Assume that $G_y$ does not have all partials equal to $0$ at the origin. Then there is some term $x_1^{i_1}\cdots x_{n - 1}^{i_{n - 1}}y^{i_n}$ in the Taylor series for $G_y$ with nonzero coefficient. Let $N = i_1 + \dots + i_n$. By Taylor's theorem, $G_y(\vec{x}, y) = P_N(\vec{x}, y) + R_N(\vec{x}, y)$, where $P_N$ is the Taylor polynomial up to degree $N$, and $R_N$ can be written in the form $$\sum_{j_1 + \dots + j_n = N} c_{j_1,\dots,j_n}(\vec{x}, y)x_1^{j_1}\cdots x_{n - 1}^{j_{n - 1}}y^{j_n},$$ where each $c_{j_1,\dots,j_n} \to 0$ as $(\vec{x}, y) \to 0$. Therefore, $\max_{r(\vec{x}, y) = r_0} G_n \sim r_0^N$ as $r_0 \to 0$.

On the other hand, when $(\vec{x}, y)$ is close to $0$, $g_y(\vec{x}, y)$ is close to $ay$, where $a = \frac{\partial g_y}{\partial y}(0)$ (which is greater than $0$); in particular, in a neighborhood of $0$, $$\frac{a}{2}y \leq g_y(\vec{x}, y) \leq 2ay.$$ This implies that $$\frac{1}{\log(e^{1/y} - \log(\frac{a}{2}))} - y \leq G_y(\vec{x}, y) \leq \frac{1}{\log(e^{1/y} - \log(2a))} - y.$$ Since $\frac{1}{\log(e^{1/y} - C)} - y \sim Cy^2/e^{\frac{1}{y}}$ as $y \to 0$, it follows that for $(\vec{x})$ in a small enough interval about $0$, $G_y(\vec{x}, y) \to 0$ faster than any polynomial as $y \to 0$.

We can use the same type of reasoning to show that $G_i$ ($1 \leq i \leq n - 1$) vanishes to all orders at $0$. If it did not, then we could find $(\vec{x}, y)$ nearby such that $G_i(\vec{x}, y)$ is only polynomially small in terms of $r(\vec{x}, y)$. We will show that this is not the case.

For any $\vec{x}$, for $y$ sufficiently close to $0$, the difference between $g_i(\vec{x}, y) - g_i(\vec{x}, 0)$ and $\frac{\partial g_i}{\partial y}(\vec{x}, 0)\cdot y$ is small relative to $y$. Indeed, if $N$ is a sufficiently small neighborhood of the origin, we will have $$|g_i(\vec{x}, y) - g_i(\vec{x}, 0) - \frac{\partial g_i}{\partial y}(\vec{x}, 0)\cdot y| < y$$ for all $(\vec{x}, y) \in N$.


Now $G_i(\vec{x}, y) = g_i(\vec{x}, \phi^2(y)) - g_i(\vec{x}, 0)$; for $(\vec{x}, y) \in N$, $$|g_i(\vec{x}, \phi^2(y)) - g_i(\vec{x}, 0) - \frac{\partial g_i}{\partial y}(\vec{x}, 0)\cdot \phi^2(y)| < \phi^2(y),$$ so $$|g_i(\vec{x}, \phi^2(y)) - g_i(\vec{x}, 0)| < \phi^2(y) + |\frac{\partial g_i}{\partial y}(\vec{x}, 0)|\cdot \phi^2(y).$$ But $|\frac{\partial g_i}{\partial y}(\vec{x}, 0)|$ is bounded above on $N$, so we have what we sought: $G_i(\vec{x}, y) \to 0$ faster than any polynomial as $(\vec{x}, y) \to 0$. This finishes the proof of claim (2).
\end{proof}

\begin{remark}
We may replace $\Diff^\infty$ with $\Diff^r$ in the statement of the theorem. Most of the proof is unchanged. We must do slightly more for claim (1), that is, showing that conjugation of a $C^r$ map by $\Phi$ is still $C^r$. This is because the map ``$h$'' may be only $C^{r - 1}.$ But it can be seen by an inductive argument that $(x, y) \mapsto h(x, \phi(y))$ is $C^r.$
\end{remark}

\section{Application: Constructing New Group Actions}
In this section, we use our main result to construct new examples of smooth actions of the discrete Heisenberg group $H$ on surfaces. This is the group of matrices 

\begin{displaymath} H =
\left\{\left( \begin{array}{ccc}
1 & a & c \\
0 & 1 & b \\
0 & 0 & 1 \end{array} \right)\colon a, b, c \in \Z \right\}.
\end{displaymath}

$H$ is generated by the elements $X = \left( \begin{array}{ccc}
1 & 1 & 0 \\
0 & 1 & 0 \\
0 & 0 & 1 \end{array} \right)$ and
$Y = \left( \begin{array}{ccc}
1 & 0 & 0 \\
0 & 1 & 1 \\
0 & 0 & 1 \end{array} \right)$. The commutator is
$Z = [X, Y] = XYX^{-1}Y^{-1} = \left( \begin{array}{ccc}
1 & 0 & 1 \\
0 & 1 & 0 \\
0 & 0 & 1 \end{array} \right)$. Since $X$ and $Y$ commute with the commutator $Z$, $H$ is a $2$-step nilpotent group. Intuitively, $H$ is ``close to abelian.'' Therefore, it is natural to consider actions of $H$ by diffeomorphisms, as a first step to extending what we know about the dynamics of diffeomorphisms and abelian groups of diffeomorphisms to the non-abelian setting. There are no relations besides two generators commuting with their commutator, so whenever we have a group with three elements $f, g,$ and $h$ of infinite order such that $[f, g] = h, fh = hf,$ and $gh = hg,$ this group is isomorphic to the Heisenberg group.

$H$ obviously acts linearly on $\R^3.$ We may ``projectivize'' this to get an action on $S^2$ which is faithful and real-analytic. Namely, let $S^2$ be the unit sphere in $\R^3.$ Define $\phi \colon H \to \Diff(S^2)$ as follows: for $x \in S^2$ and $A \in H,$ define $\phi(A)(x) = \frac{Ax}{|Ax|}.$ This action has two \textit{global fixed points}, points fixed by every element of the group: $(\pm 1, 0, 0).$

We may puncture at one of these fixed points, so we get an action of $H$ on the open disk. If we also puncture at the other fixed point, we get an action of $H$ on the open annulus. In fact, there is a canonical way of compactifying to get a closed disk $D$ or closed annulus $A,$ called ``blowing up,'' which we will describe. By doing a blow up at one or both fixed points, we get an action of $H$ on the closed disk and on the closed annulus.

\textit{Blowing up.} For reference, see Melrose \cite{Melrose}. Let us first consider blowing up the origin in $\R^n.$ Intuitively, we will remove the origin, and insert an $(n - 1)$-sphere there. Define the blow up to be $\beta(\R^n, 0) = S^{n - 1} \times [0, \infty),$ together with the blow-down map $\beta \colon \beta(\R^n, 0) \to \R^n$ given by $\beta(\theta, r) = r\theta$ ($S^{n - 1}$ is identified with the unit sphere in $\R^n$). This is a diffeomorphism of $S^{n - 1} \times (0, \infty)$ to $\R^n \setminus \{0\}.$ Given a smooth map $f \colon \R^n \to \R^n$ fixing the origin, we get an induced map $\tilde{f} \colon \beta(\R^n, 0) \to \beta(\R^n, 0),$ defined as follows:

\begin{displaymath}
\tilde{f}(x) = \left\{
     \begin{array}{cl}
     \beta^{-1}f\beta(x), & x \in S^{n - 1} \times (0, \infty)\\
     (\frac{D_0f(\theta)}{|D_0f(\theta)|}, 0), & \text{else.}
     \end{array}
   \right.
   \end{displaymath}

Here $S^{n - 1}$ is being identified with the unit tangent space of $\R^n$ at $0.$ It is a standard result that if $f$ was $C^\infty,$ then so is $\tilde{f}$; see \cite{Melrose}. Blow ups can also be done on manifolds. To blow up at $x \in M,$ take a diffeomorphism $\phi\colon U \to \R^n$ for a neighborhood $U \ni x,$ such that $\phi(x) = 0.$ Define $\beta(M, x) = ((M \setminus x) \cup S^{n - 1} \times [0, \infty))/\sim,$ where for any $y \in U \setminus x, y \sim \beta^{-1}(\phi(y)),$ $\beta$ being the blow-down defined above. $\beta(M, x)$ is a smooth manifold with boundary. For a map $f \colon M \to M$ fixing $x,$ we get a map $\tilde{f},$ as follows:

\begin{displaymath}
\tilde{f}(x) = \left\{
     \begin{array}{cl}
     f(x), & x \in M \setminus x\\
     \frac{D_0(\phi f\phi^{-1})(\theta)}{|D_0(\phi f\phi^{-1})(\theta)|}, & \text{else.}
     \end{array}
   \right.
   \end{displaymath}

Again, if $f \in \Diff^\infty(M)$ and $f$ fixes $x$, then $\tilde{f} \in \Diff^\infty(\beta(M, x)).$ Also note that if both $f$ and $g$ fix $x,$ then $\widetilde{f \circ g} = \tilde{f} \circ \tilde{g}.$ Therefore, the action of $H$ we have described on $S^2$ yields smooth actions on the closed disk and on the closed annulus, by blowing up at global fixed points. For a picture on the closed disk, see Figure 1.

The green points in the action of $Y$ have the following significance. Let $e$ (for equator) be the horizontal line through the center of the disk, and $\partial D$ the boundary; $D \setminus (e \cup \partial D)$ has two components, which we may denote $D^u$ and $D^l$ (upper and lower). For every $x \in D^u, \omega(x)$ is the right green point, where $\omega(x)$ is the $\omega$ limit set of $x.$ For every $x \in D^l, \omega(x)$ is the left green point. The action of $H$ on the closed annulus is similar; we take the action on the closed disk, and blow up at the global fixed point at the center.

Since the given actions of $H$ on $D$ and $A$ agree on the boundary components, we may glue them together, or glue the two boundary components of the annulus. The standard coordinates for the blow up may be defined by taking coordinates on the boundary sphere, and also recording $r.$ In these coordinates, the above Heisenberg actions are $C^1,$ but not $C^2.$ But by our theorem, after a change of coordinates the gluing is $C^\infty.$ See Figure 2.

In fact, we may glue arbitrarily many annuli $A_1, A_2, \dots, A_n$ concentrically, and get a smooth action of $H$ on the resulting annulus for which the $A_i$ are $H$-invariant. We may

\newpage

\begin{figure}[H]
  \centering
    \includegraphics[width=3in]{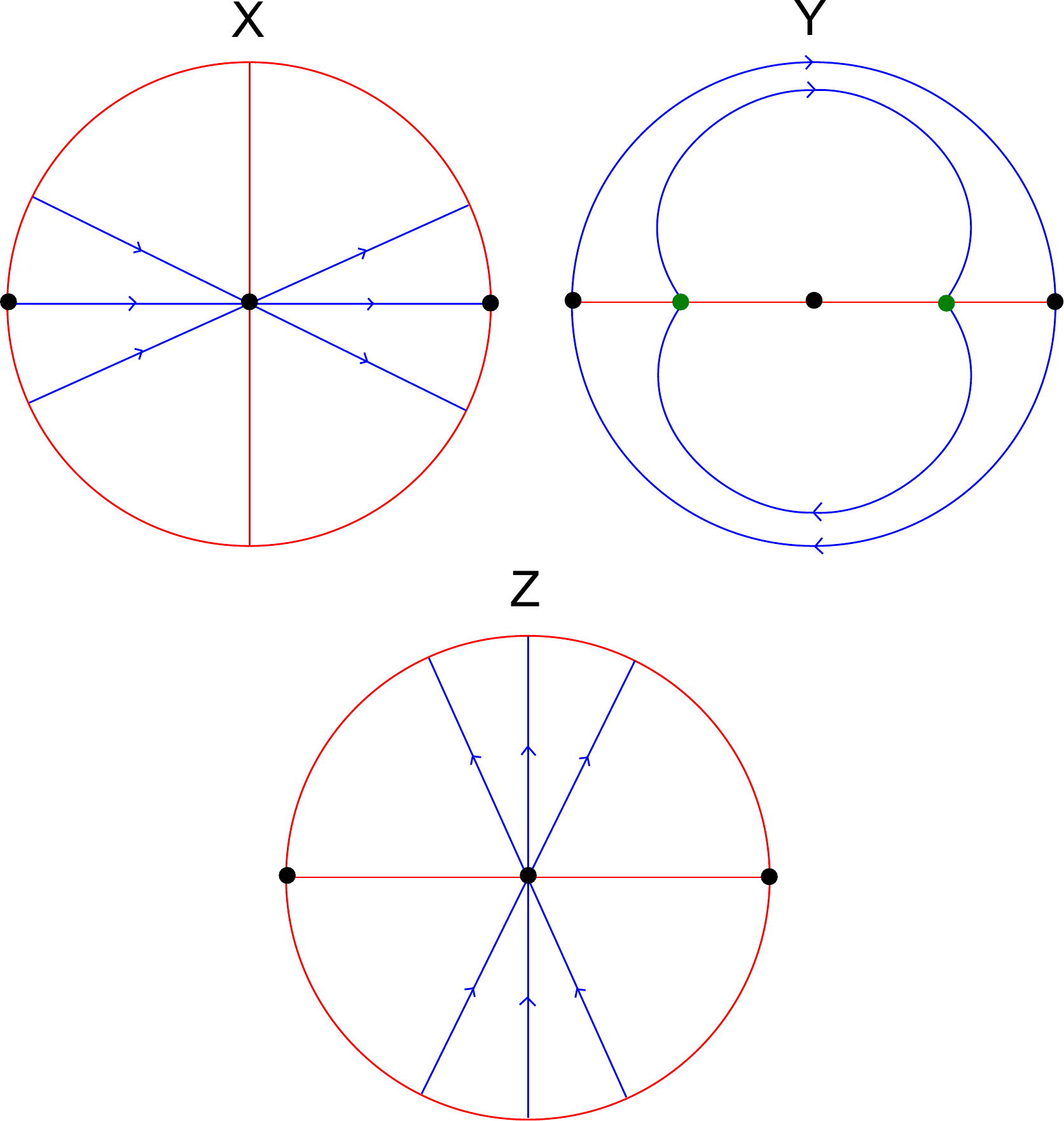}
    \caption{Action of $H$ on disk. \textcolor{Blue}{Blue:} Invariant curves; \textcolor{Red}{Red:} Fixed points; \textcolor{Black}{Black:} Fixed under whole group; \textcolor{Green}{Green:} Attractor for upper/lower half-disk}
    \vspace{12pt}
   \includegraphics[width=4.5in]{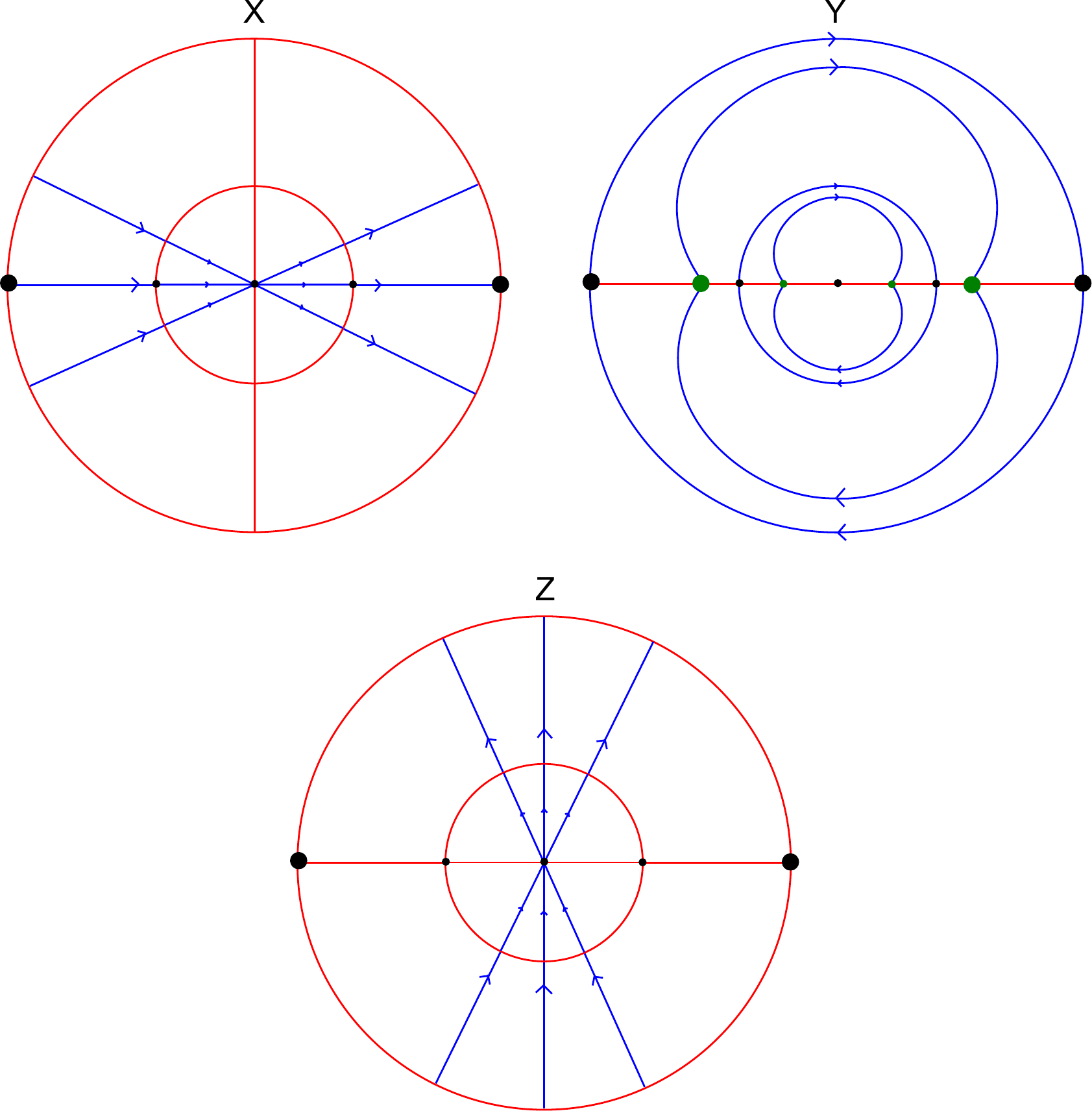}
    \caption{Action of $H$ on the disk and annulus, glued together}
\end{figure}

\noindent then glue in a disk, so we get an action on $D$; two disks, so we get an action on $S^2$; or a new annulus, so we get an action on $T^2.$ Or, we may glue the two boundary components of $\cup_i A_i,$ also yielding an action on $T^2.$ Finally, in these glued examples there are lots of global fixed points; we can blow them up and glue to get even more examples.

\begin{question}
Suppose we define an action on $N$ to be ``decomposable'' if $N$ decomposes into topological (not necessarily smooth) manifolds with boundary $M_1, M_2,$ which meet only on their boundaries, such that $M_1$ and $M_2$ are invariant. Then all the given $H$ actions are decomposable. Even the original actions on $D$ and $A$ decompose into invariant upper and lower halves, since the original action of $H$ on $S^2$ leaves the northern and southern hemispheres invariant.

Do there exist any faithful $C^\infty$ (or even $C^1$) actions of $H$ on $S^2$ which are not, in this sense, decomposable?
\end{question}

\begin{remark}
Danny Calegari \cite{Calegari} has given an example of a faithful $C^0$ action of $H$ on the sphere which is indecomposable. It is gotten as follows. First, note that in the linear action of $H$ on $\R^3,$ the plane $z = 1$ is invariant. This yields an affine action of $H$ on the plane, generated by $f(x, y) = (x + y, y),$ $g(x, y) = (x, y + 1)$ with commutator $h(x, y) = (x + 1, y).$

Let $\alpha$ be irrational, and let $T_{(\alpha, 0)}$ denote horizontal translation by $\alpha.$ We may take the quotient of the plane by $T_{(\alpha, 0)}$; the result is a cylinder. The action of $H$ on the plane descends to this cylinder, since both $f$ and $g$ commute with $T_{(\alpha, 0)}.$ This action is smooth. We may next compactify the two ends of the cylinder with points, so we get a sphere. The resulting action of $H$ on $S^2$ is only $C^0,$ and it cannot be conjugated to a $C^1$ example.

In this example, $f$ acts by twisting by an amount dependent on how high up we are. As we move up towards the north pole, or down to the south pole, it twists infinitely many times. $g$ acts by moving every point except the poles upward. $h$ is an irrational rotation of the sphere. It is easily seen that this example is indecomposable in the sense we have given. However, it does not satisfy the stronger condition of topological transitivity, namely, that there exist a point $x \in S^2$ whose orbit under $H$ is dense.
\end{remark}

\begin{question}
Does there exist a faithful $C^0$ action of $H$ on $S^2$ which is topologically transitive?
\end{question}

\begin{remark}
There does exist a faithful $C^\infty$ action of $H$ on the torus $T^2$ which is topologically transitive, indeed, which is minimal. Consider the action of $H$ on the plane given above. Let $\alpha$ be irrational. Let $T_{(\alpha, 0)}$ and $T_{(0, \alpha)}$ be irrational horizontal and vertical translations; if we take the quotient $\R^2/\langle T_{(\alpha, 0)}, T_{(0, \alpha)}\rangle$ we get a torus. We claim that the action of $H$ on the plane descends to this torus. Since $g$ is a translation, it commutes with $T_{(\alpha, 0)}$ and $T_{(0, \alpha)}.$ The skew map $f$ commutes with $T_{(\alpha, 0)}.$ It fails to commute with $T_{(0, \alpha)},$ but the commutator is $[f, T_{(0, \alpha)}] = T_{(\alpha, 0)},$ which is trivial on the torus.

If we denote the induced maps on the torus with bars, then $\bar{f}$ and $\bar{h}$ are vertical and horizontal irrational translations of the torus. Thus the orbit of any point is dense under this action of $H.$
\end{remark}

\section{Application: Generalizing a Result of Franks-Handel}
A useful notion from geometric group theory is that of \textit{distortion}.

\begin{definition}
If $G$ is a finitely generated group, and we choose the generating set $\{g_1, \dots, g_s\}$, then $f \in G$ is said to be a \textit{distortion element} of $G$ if $f$ has infinite order and

$$\liminf_{n \to \infty} \frac{|f^n|}{n} = 0,$$

where $|f^n|$ is the word length of $f^n$ in the generators $\{g_1, \dots, g_s\}$.
\end{definition}

\begin{remark}
We could have taken the liminf to be a limit; because word length is subadditive, the limit must exist, and is called the \textit{translation length} (see \cite{G&S}).

It is straightforward to see that the property of being a distortion element is independent of the finite generating set chosen. If $G$ is not finitely generated, we say that $f \in G$ is distorted in $G$ if it is distorted in some finitely generated subgroup of $G$.
\end{remark}

See Gromov \cite{Gromov} for a good discussion with many examples. One reason why distortion elements are interesting is that well-known groups have them. It is easy to check that the central elements of the three-dimensional Heisenberg group are distortion elements. Lubotzky, Mozes, and Raghunathan proved that irreducible nonuniform lattices in higher-rank Lie groups have distortion elements (\cite{LMR}).

\vspace{12pt}

Franks and Handel \cite{F&H} proved the following

\begin{theorem}
Let $S$ be a closed surface. Let $\Diff^1(S)_0$ denote $C^1$ diffeomorphisms isotopic to the identity. Let $f \in \Diff^1(S)_0$ be distorted. If $S = S^2,$ assume $f$ has at least 3 fixed points; if $S = T^2,$ assume $f$ has at least 1 fixed point. Then \newline(*) For any $f$-invariant Borel probability measure $\mu, \supp(\mu) \subset \Fix(f).$
\end{theorem}

In particular, (*) says that $f$ cannot be area-preserving, since the support of area is the whole surface, so $f$ would be the identity, which is not a distortion element.

\begin{remark}
The assumption on fixed points is necessary. Calegari and Freedman \cite{C&F} showed that an irrational rotation of $S^2$ or $T^2$ is a distortion element.
\end{remark}

Using our theorem, we have the following corollary of Franks and Handel's result:

\begin{corollary}
Let $S$ be a compact surface with nonempty boundary. Let $f \in \Diff^1(S)_0$ be distorted. If $S = D$ is the closed disk, assume that $f$ has at least 2 fixed points, at least one of which is not on $\partial D.$ If $S = A,$ assume $f$ has at least one fixed point. Then (*) holds.
\end{corollary}

\begin{proof}
Assume, by way of contradiction, that (*) does not hold. Thus there is some $f$-invariant measure $\mu$ such that $\supp(\mu) \not\subset \Fix(f).$ Let $M_i = S \times \{i\}$ for $i = 1, 2.$ Glue each boundary component of $M_1$ to the corresponding one in $M_2,$ in the obvious way: for $x \in \partial S, \alpha(x, 1) = (x, 2).$ Choose a smooth structure on the result, as described above. The result can be called the \textit{double} of $S, D(S).$

By Remark 2 following the theorem, we can choose conjugacies $\Psi_1$ and $\Psi_2$ to make the glued map $\tilde{f} = \mathfrak{g}(\Psi_1^{-1}f_1\Psi_1, \Psi_2^{-1}f_2\Psi_2)$ be $C^1.$ In fact, since $M_1$ and $M_2$ are both copies of $S,$ we can find a single conjugacy $\Psi \in \Homeo(S)$ and let $\Psi_i = \Psi \times id, i = 1, 2.$

Any fixed point for $f$ in the interior $\Int(S)$ yields two fixed points for $\tilde{f},$ one on each side. A fixed point for $f$ on $\partial S$ yields one fixed point for $\tilde{f}.$ The hypotheses given in the corollary guarantee that $\tilde{f}$ has as many fixed points as the theorem of Franks and Handel requires: one for $T^2$ and three for $S^2.$

There will be an invariant probability measure not supported on the fixed point set of $\tilde{f}.$ Namely, $\Psi^{-1}f\Psi$ preserves the measure $\Psi^{-1}_*(\mu),$ whose support is $\supp(\Psi^{-1}_*(\mu) = \Psi^{-1}(\supp(\mu)).$ On the other hand, $\Fix(\Psi^{-1}f\Psi) = \Psi^{-1}(\Fix(f)).$ So $\supp(\Psi^{-1}_*(\mu)) \not\subset \Fix(\Psi^{-1}f\Psi).$ If we take $\Psi^{-1}_*(\mu)$ on both copies of $S,$ we get an invariant measure for $\tilde{f}$ not supported in $\Fix(\tilde{f}).$

Therefore, by Franks and Handel, $\tilde{f}$ is not distorted in $\Diff^1(D(S))_0.$ But the map $f \mapsto \tilde{f}$ is a 1-1 homomorphism, so it preserves the property of being a distortion element, a contradiction.
\end{proof}

\begin{remark}
We can slightly weaken the assumptions: If $S = D,$ and a distortion element $f$ has at least 2 fixed points anywhere, or has a fixed point on the boundary, then (*) holds. Possibly even this fixed point assumption can be dropped. So we have the following question:
\end{remark}

\begin{question}
Is an irrational rotation of the disk distorted in $\Diff^1(D)_0$?
\end{question}


\begin{thebibliography}{9}
\bibitem{Calegari}
D. Calegari, personal communication with John Franks, as cited in J. Franks, \emph{Distortion in groups of circle and surface diffeomorphisms}, Dynamique des diff\'{e}omorphismes conservatifs des surfaces: Un point de vue topologique, Panoramas et Synth\`{e}ses \textbf{21} (2006), Socie\'{e}t\'{e} Math\'{e}matique de France, 2006.

\bibitem{C&F}
D. Calegari and M. H. Freedman, \emph{Distortion in transformation groups}, Geom. Topol. \textbf{10} (2006), 267-293.

\bibitem{F&S}
B. Farb and P. Shalen, \emph{Lattice actions, 3-manifolds and homology}, Topology \textbf{39}(3) (2000), 573-587.

\bibitem{F&H}
J. Franks and M. Handel, \emph{Distortion elements in group actions on surfaces}, Duke Math. J. \textbf{131} (2006), 441-468.

\bibitem{G&S}
S. M. Gersten and H. B. Short, \emph{Rational subgroups of biautomatic groups}, Annals of Math. (2) \textbf{134} (1991), 125-158.

\bibitem{Gromov}
M. Gromov, \emph{Asymptotic invariants of infinite groups}, Vol. 2 of Geometric Group Theory: Proceedings of the Symposium at Sussex University (1991), G. Niblo and M. Roller eds., London Math. Soc. Lecture Note Ser. \textbf{182}, Cambridge University Press, Cambridge, 1993.

\bibitem{K&L}
A. Katok and J. Lewis, \emph{Global rigidity results for lattice actions on tori and new examples of volume-preserving actions}, Israel J. Math. \textbf{93} (1996), 253-280.

\bibitem{Kloeckner}
B. Kloeckner, \emph{Almost homogeneous manifolds with boundary}, Trans. Amer. Math. Soc. \textbf{361} 12 (2009), 6729-6740.

\bibitem{LMR}
A. Lubotzky, S. Mozes, and M. S. Raghunathan, \emph{The word and Riemannian metrics on lattices of semisimple groups}, Inst. Hautes \'{E}tudes Sci. Publ. Math. \textbf{91} (2000), 5-53.

\bibitem{Melrose}
R. Melrose, \emph{Differential analysis on manifolds with corners}, Chapter 5, \newline\newblock{\tt http://www-math.mit.edu/~rbm/book.html}.

\bibitem{Tsuboi}
T. Tsuboi, \emph{$\Gamma_1$-structures avec une seule feuille}, Ast\'{e}risque \textbf{116} (1984), 222-234.
\end{thebibliography}
\end{document}